\documentclass[a4paper, 11pt]{article}
\usepackage{amsmath, amssymb,amscd, amsthm}
\usepackage[mathscr]{eucal}
\usepackage{graphics}
\usepackage{fullpage}
\usepackage{url}
\newcommand\cyr{%
 \renewcommand\rmdefault{wncyr}%
 \renewcommand\sfdefault{wncyss}%
 \renewcommand\encodingdefault{OT2}%
\normalfont\selectfont} \DeclareTextFontCommand{\textcyr}{\cyr}

\newtheorem{theorem}{Theorem}
\newtheorem{lemma}[theorem]{Lemma}
\newtheorem{corollary}[theorem]{Corollary}
\newtheorem{proposition}[theorem]{Proposition}
\newtheorem{remark}[theorem]{Remark}

\newtheorem{claim}[theorem]{Claim}

\def\mod{\operatorname{mod}}

\long\def\symbolfootnote[#1]#2{\begingroup%
\def\thefootnote{\fnsymbol{footnote}}\footnote[#1]{#2}\endgroup}

\title{New Restriction Estimates for the 3-d Paraboloid over Finite Fields}

\author{Mark Lewko\thanks{Supported by a NSF Postdoctoral Fellowship, DMS-1204206.}}

\date{}

\begin{document}

\maketitle

\begin{abstract} We improve the range of exponents for the restriction problem for the 3-d paraboloid over finite fields. The key new ingredient is a variant of the Bourgain-Katz-Tao finite field incidence theorem derived from sum-product estimates. In prime order fields, we give an explicit relationship between the exponent in this incidence theorem and restriction estimates for the paraboloid.
\end{abstract}

\symbolfootnote[0]{2010 Mathematics Subject Classification. 42B20}

\section{Introduction}Given a hypersurface $ S \subset \mathbb{R}^n$ with measure $d\sigma$, the (Euclidean) restriction problem asks one to determine for which $p$ and $q$ does the following inequality hold:
\[ ||(fd \sigma)^{\vee} ||_{L^q(\mathbb{R}^n)} \ll_{p,q} ||f ||_{L^{p}(S, d\sigma)}.\]
If one takes $S$ to be the the sphere or paraboloid and $n \geq 3$, this is a challenging open problem. We refer the reader to \cite{T} for a more thorough account of work related to these problems (through 2003).

In 2002, Mockenhaupt and Tao formulated a finite field analogue of this problem. Let $\mathbb{F}$ denote a finite field, and $\mathbb{F}^n$ the $n$-dimensional cartesian product of $\mathbb{F}$. We let $e(\cdot)$ denote a non-principal character on $\mathbb{F}$ and for $x=(x_1,\ldots,x_n),\xi=(\xi_1,\ldots,\xi_n) \in \mathbb{F}^n$ we define the dot product by $x\cdot \xi= x_1 \xi_1 + x_2 \xi_2 + \ldots +x_n \xi_n$.  We endow the vector space $\mathbb{F}^{n}$ with the counting measure $dx$, and its (isomorphic) dual space $\mathbb{F}_{*}^n$ with the normalized counting measure $d\xi$ that assigns measure $|\mathbb{F}^n|^{-1}$ to each point. For a complex-valued function $f$ on $\mathbb{F}^n$ the $L^p$ norm is given by $ ||f||_{L^p(\mathbb{F}^n,dx)} = \left(\sum_{x \in \mathbb{F}^n} |f(x)|^p \right)^{1/p} $ and its Fourier transform (defined on the dual space $\mathbb{F}^n_{*}$)
\[ \hat{f}(\xi) = \sum_{x \in \mathbb{F}^n} f(x) e(-x \cdot \xi)dx.\]
Given a non-empty set $S \subset \mathbb{F}^n_{*}$, we may define the measure (`normalized surface area') $d\sigma$ which assigns the measure $|S|^{-1}$ to each point so that
\[ \sum_{\xi \in S} g(\xi) d\sigma := \frac{1}{|S|}\sum_{\xi \in S} g(\xi)\]
and the inverse Fourier transform on $S$ is given by
\[  (gd \sigma)^{\vee}(x) = \frac{1}{|S|} \sum_{\xi\in S} g(\xi)e(x \cdot \xi ).\]

For $1 \leq p,q \leq \infty$, we define $\mathcal{R}^{*}(p\rightarrow q)$ to be the best constant such that, for all $g$ supported on $S$, we have
\[|| (gd\sigma)^{\vee}||_{L^q(\mathbb{F}^{n},dx) } \leq \mathcal{R}^{*}(p\rightarrow q) ||g||_{L^{p}(S,d\sigma)}.\]
By duality, this is also equal to the best constant in the inequality
\[||\hat{f}||_{L^{p'}(S,d\sigma)} \leq  \mathcal{R}^{*}(p\rightarrow q)||f||_{L^{q'}(\mathbb{F}^n,dx) }.\]
When $S$ is an algebraic variety (and hence has a natural interpretation over every finite field) the restriction problem for $S$ seeks to classify the pairs of exponents $(p,q)$ for which $\mathcal{R}^{*}(p\rightarrow q)$ is bounded independent of the field size. In the finite field setting, the restriction phenomenon has already received a fair amount of attention, see \cite{IK}, \cite{IKs}, \cite{IKq} \cite{Koh1}, \cite{lewko}, \cite{MT}.

Here we will be interested in the case of the $3$ dimensional paraboloid, $S=\{(\omega,\omega \cdot \omega)  :\omega \in \mathbb{F}_{*}^2  \}$ (which we'll denote as $\mathcal{P}$). The Stein-Tomas method (described in the next section) gives $\mathcal{R}^{*}(2 \rightarrow 4) \ll 1$. In general, this is the best one can hope for with $p=2$, which can be seen by testing the extension operator $(fd\sigma)^{\vee}$ on the characteristic function of the affine subspace $ (\xi,i \xi, 0) \subset \mathcal{P} $ if $-1=i^2$ is the square of an element $i \in \mathbb{F}$. If such an $i$ does not exist (that is the characteristic of the field is $p = 3 \mod 4$), then it is conjectured that $\mathcal{R}^{*}\left (2 \rightarrow 3 \right) \ll 1$. In this direction, Mockenhaupt and Tao were able to show that $\mathcal{R}^{*}\left (2 \rightarrow \frac{18}{5}+\epsilon \right) \ll 1$. This has been improved (see \cite{lewko}) to $\mathcal{R}^{*}(2 \rightarrow \frac{18}{5}) \ll 1$ using a bilinear version of the Mockenhaupt and Tao argument (this was first proved by Bennett, Carbery, Garrigos, and Wright in unpublished work). Apart from this endpoint, however, the exponent of $\frac{18}{5}$ has stood for over a decade. Our first result is the following:

\begin{theorem}\label{thm:general}Let $\mathbb{F}$ be a field such that $-1$ is not a square. There exists a $\delta >0$ such that
$\mathcal{R}^{*}(2 \rightarrow \frac{18}{5} - \delta) \ll 1$.
\end{theorem}
In the Euclidean setting, it is known that the restriction problem is closely connected with the Kakeya problem in geometric measure theory. Roughly speaking, the Kakeya (family of) problems seek to quantify the extent to which tubes in Euclidean space can overlap. Indeed, an affirmative solution to the restriction problem implies an affirmative solution to the Kakeya (maximal) problem. Conversely, much of the progress on the Euclidean restriction problem has made use of progress on the Kakeya (maximal) problem. We refer the reader to \cite{T} for a more detailed discussion.  In the finite field setting, however, there seems to be less of a connection between the restriction and Kakeya problems. Indeed, in 2008 Dvir \cite{Dvir2} proved the full finite field Kakeya conjecture, however neither his result nor its proof technique have yet yielded any progress on the finite field restriction problem.

\textbf{Note added}: The author recently found a formal connection between the finite field Kakeya and restriction conjectures. Indeed, it can be shown that the finite field restriction conjecture (in fields in which $-1$ is a square) implies the finite field Kakeya conjecture (Dvir's theorem). Moreover, the finite field Kakeya maximal operator estimates of Ellenberg, Oberlin and Tao \cite{EOT} proven using Dvir's polynomial method can in turn be used to make progress on the finite field restriction conjecture. See \cite{lewkoKakeya} for details.

The main contribution of our current work is to connect the finite field restriction problem not with the finite field Kakeya problem, but with the related finite field Szemer\'{e}di-Trotter incidence problem. Thus, the main new ingredient here will be (a variant of) the finite field Szemer\'{e}di-Trotter-type theorem of Bourgain, Katz, and Tao \cite{BKT}. Let $P$ be a set of points and $L$ a set of lines in $\mathbb{F}^2$. We define the set of incidences as
\[I(P,L) := \{ (x,\ell) \in P \times L : x \in \ell\}.\]
An easy and well-known argument gives $\left|I(P,L)\right| \leq \min\left( |P|^{1/2}|L| +|P|, |P||L|^{1/2}  +|L| \right).$ In their breakthrough paper \cite{BKT}, Bourgain, Katz, and Tao proved that if $\mathbb{F}$ is a prime order finite field and $N = |\mathbb{F}|^\beta$ for $0<\beta < 2$ then for any set of points $P$ and lines $L$ such that $|L|=|P|=N$ one may improve the previous estimate to:
\[ \left|I(P,L)\right| \ll N^{3/2-\epsilon}\]
for some $\epsilon(\beta)>0$.  If we wish  to quantify the exponent in our restriction theorem, we will need a quantitative version of the above result. Here, we introduce the notation $\mathcal{I}(\alpha,\beta) $
to denote the claim that if $N \leq |\mathbb{F}|^\beta$ then for all sets of points $P$ and $L$ such that $|P|=|L|=N$, we have
\[ \left|I(P,L)\right| \ll_{\alpha,\beta} N^{\alpha}.\]
In prime order finite fields, we give the following relationship between $\mathcal{I}(\alpha,\beta) $ and restriction estimates:
\begin{theorem}\label{thm:incTores}Let $\mathbb{F}$ be a prime order finite field such that $-1$ is not a square. Furthermore, assume $\mathcal{I}\left( \alpha, \frac{8(\alpha -3)(\alpha -2) }{9\alpha -6  }\right)$. Then
\[\mathcal{R}^{*}\left(2\rightarrow \frac{12 - 2 \alpha}{4-\alpha} + \epsilon \right) \ll 1  \]
for all $\epsilon >0$.
\end{theorem}
It appears the best result to date is $\mathcal{I}(\frac{3}{2}-\frac{1}{662}+\epsilon, 1)$ for all $\epsilon >0$ due to Jones \cite{Jones3} (which improves on prior explicit estimates of Helfgott and Rudnev \cite{HR} and Jones \cite{Jones1}). Taking $\alpha = \frac{3}{2}-\frac{1}{662} = \frac{496}{331}$, we have that  $\frac{8(\alpha -3)(\alpha -2) }{9\alpha -6  } = \frac{47144}{68587} \leq .805 \leq 1$. Thus, we derive the following restriction theorem:
\begin{corollary}Let $\mathbb{F}$ be a prime order finite field such that $-1$ is not a square. Then,
$$\mathcal{R}^{*}\left(2\rightarrow \frac{18}{5} - \frac{1}{1035} + \epsilon\right) \ll 1$$
for any $\epsilon >0$.
\end{corollary}
We note (see Remarks \ref{rem:alpha1} and \ref{rem:alpha2} below) that the conclusion of Theorem \ref{thm:incTores}, does not require the full strength of the claim $\mathcal{I}\left( \alpha, \frac{8(\alpha -3)(\alpha -2) }{9\alpha -6  }\right)$, but rather only $\mathcal{I}\left( \alpha, \beta \right)$ for $\beta =\frac{2}{4-\alpha}$ and then progressively weaker values of $\alpha$ as $\beta$ increases to $\frac{8(\alpha -3)(\alpha -2) }{9\alpha -6  }$ (at which point the trivial incidence theorem is sufficient). We avoid such a formulation for simplicity. One may check that if one can take $\alpha = \frac{4}{3}$ (which seems the best one can hope for) for sets of size $\leq |\mathbb{F}|^{\frac{40}{27}}$ (or, as noted, some weaker variant of this), then we would obtain
$$\mathcal{R}^{*}\left(2\rightarrow \frac{7}{2}+ \epsilon\right) \ll 1,$$
which still falls short of the full conjectured estimate $\mathcal{R}^{*}\left(2\rightarrow 3\right) \ll 1$. In the fourth section we show how to use a more complicated incidence theorem in fields not of prime order to obtain Theorem 1. Here we do not quantify the argument to produce an explicit value of $\delta$.
\section{The Stein-Tomas method}
In this section we will reprove the Stein-Tomas theorem from \cite{MT}, and derive some related estimates which we will need later. First recall the notion of Fourier dimension. Given $S \subset \mathbb{F}^n_{*}$ with normalized surface measure $d\sigma$, the inverse Fourier transform of $d\sigma$ is given by
\[(d\sigma)^{\vee}(x) = \frac{1}{|S|}\sum_{\xi \in S} e(x\cdot \xi).\]
Note that $(d\sigma)^{\vee}(0)=1$, however for certain $S$ we may hope that $|(d\sigma)^{\vee}(x)|$ is small for $x\neq 0$. In particular, we define the (Fourier) dimension of $S$ to be the largest $\tilde{d}$ such that
\[|(d\sigma)^{\vee}(x)| \leq |\mathbb{F}|^{-\tilde{d}/2}  \]
for all $x \neq 0$.  It is also convenient to define the Bochner-Riesz kernel $K$ associated to $S$ by $K(x) = (d\sigma)^{\vee}(x) - \delta_{0}(x)$ (where the delta function $\delta_{0}$ is defined to be $1$ at $0$ and $0$ otherwise). We will use the following well-known fact (see \cite{MT}) which follows from elementary Gauss sum estimates:
\begin{proposition}The Fourier dimension of the $3$-dimensional paraboloid is $2$. That is,
\[|(d\sigma)^{\vee}(x)| \ll |\mathbb{F}|^{-1} \]
for $x \neq 0$.
\end{proposition}
We are now ready to revisit the Stein-Tomas argument. Let $f:\mathbb{F}^n \rightarrow \mathbb{C}$ and $S$ and $d\sigma$ be as above.
\begin{lemma}\label{Lem:STinfty}Let $p,q \geq 2$ and $0\leq \theta \leq 1$. Furthermore, let $||f||_{L^{\infty}} \leq \lambda$ and $||f||_{L^{(q/\theta)'}}=1$ ,then
\[ ||\hat{f}||_{L^{p'}(S,d\sigma)} \leq \mathcal{R}^{*}(p\rightarrow q) \lambda^{(1-\theta)/(q-\theta)}.\]
\end{lemma}
\begin{proof}Letting $||f||_{L^{(q/\theta)'}}=1$ and using that $(q/\theta)' = \frac{q}{q-\theta}$ we have:
\[||f||_{q'} =  (\sum_{x \in \mathbb{F}^n} |f(x)|^{q/(q-1)})^{(q-1)/q} \leq \left(\sum_{x\in \mathbb{F}^n} |f(x)|^{q/(q-\theta)} \lambda^{q/(q-1) - q/(q-\theta)} \right)^{(q-1)/q} \leq \lambda^{(1-\theta)/(q-\theta)},\]
Thus
\[||\hat{f}||_{L^{p'}(S,d\sigma)} \leq \mathcal{R}^{*}(p \rightarrow q)||f||_{L^{q'}}  \leq \mathcal{R}^{*}(p \rightarrow q) \lambda^{(1-\theta)/(q-\theta)}\]
which implies the lemma.
\end{proof}
For convenience, let us also record the following related result:
\begin{lemma}\label{lem:charinfST}Let $1/2 \leq |f|\leq 1$ on its support $E$ satisfying $|E|=|\mathbb{F}|^{\gamma}$, and $\mathcal{R}^{*}(p\rightarrow q) \ll |\mathbb{F}|^{\alpha}$. Then
\[ ||\hat{f}||_{L^{p'}(S,d\sigma)} \ll ||f||_{L^{\frac{q\gamma}{q\gamma - \gamma + \alpha q}}(\mathbb{F}^3,dx)}.\]
\end{lemma}
\begin{proof}We have
\[ ||\hat{f}||_{L^{p'}(S,d\sigma)} \ll  |\mathbb{F}|^{\alpha} ||f||_{L^{q'}(\mathbb{F}^3,dx)} \ll |\mathbb{F}|^{\alpha} |\mathbb{F}|^{\frac{(q-1) \gamma}{q}}\]
\[= |\mathbb{F}|^{\gamma(1-1/q)+\alpha} \ll ||f||_{L^{\frac{q\gamma}{q\gamma - \gamma + \alpha q}}(\mathbb{F}^3,dx)}. \]

\end{proof}

\begin{lemma}\label{Lem:STsupp}Let $p,q \geq 2$ and $0\leq \theta \leq 1$. Let $ |f| \geq \lambda$ and $||f||_{L^{(q/\theta)'}}=1$ then
\[||\hat{f}||_{L^{p'}(S,d\sigma)} \ll 1+  ||K||_{L^{\infty}(\mathbb{F}^n, dx)}^{1/2} \lambda^{-\theta/(q-\theta)}.\]
\end{lemma}
\begin{proof}We decompose $(d\sigma)^{\vee} = \delta_{0}+ K$ where $\delta_{0}(x)$ is $1$ for $x=0$ and $0$ otherwise. By Plancherel we have
\[||\hat{f}||_{L^{p'}(S,d\sigma)}^2 \leq ||\hat{f}||_{L^{2}(S,d\sigma)}^2 \leq |\left<f, f*(d\sigma)^{\vee} \right>|.\]
By Young's inequality we then have
\[\leq ||f||_{L^2(\mathbb{F}^n, dx) }^2 + |F|^{-\tilde{d}/2}||f||_{L^1(\mathbb{F}^n,dx)}^{2}.\]
Notice that $||f||_{L^2(\mathbb{F}^3, dx) } \leq ||f||_{L^{(q/\theta)'}}=1$. Now $|\text{supp}(f)|^{(q-\theta)/q} \lambda \leq 1$ so $|\text{supp}(f)| \leq \lambda^{-q/(q-\theta)}$
\[ \sum_{x \in \mathbb{F}^n} |f(x)| \leq (\sum_{x \in \mathbb{F}^n} |f(x)|^{q/(q-\theta)})^{(q-\theta)/q} |\text{supp}(f)|^{\theta/q} \leq \lambda^{-\theta/(q-\theta)}.  \]
Thus we have
\[||\hat{f}||_{L^{p'}(S,d\sigma)} \ll 1+  ||K||_{L^{\infty}(_{\mathbb{F}^n}, dx)}^{1/2} \lambda^{-\theta/(q-\theta)}   \]
which completes the proof of the claim.
\end{proof}
Taking $\lambda= | \mathbb{F}|^{-\tilde{d}(q-\theta)/4} (\mathcal{R}^{*}(p \rightarrow q))^{-(q-\theta)}$ in the two lemmas above, we conclude the formulation of Stein-Tomas theorem given in \cite{MT}:
\begin{lemma}\label{lem:mtST}Let $p,q \geq 2$ and assume that $S$ has Fourier dimension $\tilde{d} >0$. Then for any $0 < \theta < 1$ we have that
\[\mathcal{R}^{*}(p \rightarrow q/\theta) \ll 1 +  \mathcal{R}^{*}(p \rightarrow q)^{\theta} |\mathbb{F}|^{\frac{-\tilde{d}(1-\theta)}{4}}.\]
\end{lemma}
We also find it useful to use the following consequence of Lemma \ref{Lem:STsupp}:
\begin{corollary}\label{cor:stdecay}Let $1/2 \leq |f|\leq 1$ on its support $E \subseteq \mathbb{F}^3$ satisfying   $|E|=|\mathbb{F}|^\gamma$.  Then
\begin{equation}\label{eq:decay}
 ||\widehat{f}||_{L^{2}(\mathcal{P},d\sigma)} \ll  ||1_{E} ||_{L^{2}(\mathbb{F}^n, dx)} +||1_{E} ||_{L^{\frac{2\gamma}{2\gamma-1}}(\mathbb{F}^n, dx)}.
 \end{equation}
\end{corollary}
In addition, we will use the following consequence of Lemma \ref{lem:mtST}:
\begin{lemma}\label{lem:eprem}($\epsilon$-removal lemma) Let $S$ have Fourier dimension $\tilde{d}>0$ and assume that $\mathcal{R}^{*}(p\rightarrow q) \ll_{\epsilon} |\mathbb{F}|^{\epsilon}$, then $\mathcal{R}^{*}(p \rightarrow q+\delta) \ll_{\delta} 1$.
\end{lemma}
The strategy of \cite{MT} proceeds by first proving the following local restriction estimate $\mathcal{R}^{*}(2 \rightarrow 16/5) \ll_{\epsilon} |\mathbb{F}|^{1/16 + \epsilon}$. In turn, a key ingredient in the proof of this local restriction estimate, is the estimate $\mathcal{R}^{*}(8/5 \rightarrow 4)\ll_{\epsilon}|\mathbb{F}|^{\epsilon}$ which is proved by expanding out an $L^4$ norm (in the dual extension formulation) and using some combinatorics (we will revisit this argument shortly). This has been improved to $\mathcal{R}^{*}(8/5 \rightarrow 4)\ll 1$ (see \cite{lewko}, although this was first proved by Bennett, Carbert, Garrigos and Wright in unpublished work) which when combined with the argument of \cite{MT} gives the slightly stronger local estimate:
\begin{proposition}\label{prop:localres}Let $S$ be the $3$-d paraboloid. Then:
\[\mathcal{R}^{*}(2 \rightarrow 16/5) \ll  |\mathbb{F}|^{1/16}.\]
\end{proposition}
From this estimate, one may then recover $\mathcal{R}^{*}(2 \rightarrow \frac{18}{5})$ by applying Lemma \ref{lem:mtST} with $p=2$, $q=16/5$, $\tilde{d}=2$ and $\theta=8/9$. One can view our current approach to improving the $\mathcal{R}^{*}(2 \rightarrow 18/5)$ result as isolating the bottleneck in the argument just described. This turns out to occur when the function $f$ the restriction operator is being applied to is essentially the characteristic function of a set of dimension $9/5$ (and has some additional regularity properties). We are then able to improve the $L^4$ analysis in the case of functions of this form by inserting a nontrivial incidence estimate into the argument.
\section{The $L^4$ estimate revisited}
As previously mentioned, the following extension estimate was proven in \cite{lewko}:
\[||(fd\sigma)^{\vee}||_{L^4(\mathbb{F},dx)} \leq ||f||_{L^{8/5}(\mathcal{P},d\sigma)} .\]
where $\mathcal{P}$ is the $3$-dimensional paraboloid and $d \sigma$ is the normalized surface measure. This is sharp in the sense that the $L^{8/5}$ norm can't be replaced with a smaller $L^p$ norm. However, we will show that if the the support of $f$ is contained in a set that satisfies a nontrival incidence estimate then this estimate can be improved somewhat. More precisely:
\begin{proposition}\label{prop:L4}Let $E \subset \mathbb{F} \times \mathbb{F}$ such that for all sets of points $A \subset \mathbb{F} \times \mathbb{F}$ and lines $L$ of comparable size (say, $\frac{1}{2} |E| \leq |L|, |A| \leq 2 |E|$) one has $|I(A, L)| \ll |A|^{\alpha}$, then
\[||(f d\sigma)^\vee ||_{L^4( \mathbb{F}^3, dx)} \ll |E|^{(1+\alpha)/4}|F|^{-5/4}\]
for any $|f| \leq 1_{E}$.
\end{proposition}
Note that using the trivial incidence inequality $|I(A,L)| \ll |A|^{3/2} $ (for $|A|\sim|L|$) allows one to recover $||(1_{E} d\sigma)^\vee ||_{L^4(\mathbb{F}^3, dx)} \ll  || 1_{E}||_{L^{8/5}(S,d\sigma)} $.  The proof closely follows the exposition of the $L^4$ estimate in \cite{lewko}, with adjustments to involve the incidence hypothesis.
\begin{proof}Expanding the $L^4$ norm (as we will do below) one easily sees that $||(f d\sigma)^\vee ||_{L^4( \mathbb{F}^3, dx)}  \leq  ||(1_{E} d\sigma)^\vee ||_{L^4( \mathbb{F}^3, dx)}  $. Thus we can replace $f$ with $1_{E}$ throughout the following.  Expanding the $L^4$ norm and using the definition of $(1_E d\sigma)^\vee$, we have:
\[\left|\left|(1_E d\sigma)^\vee (1_E d\sigma)^\vee \right|\right|_{L^2(\mathbb{F}^3,dx)}^2 =  \sum_{x \in \mathbb{F}^3} \left|(1_E d\sigma)^\vee (x) (1_E d\sigma)^\vee (x)\right|^2.\]
\[ = \sum_{x \in \mathbb{F}^3} \left| \frac{1}{|\mathcal{P}|} \sum_{\xi_1 \in \mathcal{P}} \chi_A(\xi_1) e(x\cdot \xi_1) \cdot \frac{1}{|\mathcal{P}|} \sum_{\xi_2 \in \mathcal{P}} \chi_B(\xi_2) e(x\cdot \xi_2)\right|^2.\]
We can rewrite this as:
\[\frac{1}{|\mathcal{P}|^4} \sum_{x\in  \mathbb{F}^3} \left| \sum_{\xi_1 \in \mathcal{P}} 1_E(\xi_1) e(x \cdot \xi_1) \cdot \sum_{\xi_2 \in \mathcal{P}} 1_E (\xi_2) e(x\cdot \xi_2)\right|^2\]
\[ = \frac{1}{|\mathcal{P}|^4} \sum_{x \in \mathbb{F}^3} \sum_{a,b,c,d \in \mathcal{P}} 1_E(a) 1_E(b) 1_E(c) 1_E(d) e(x\cdot a) e(x \cdot b) e(-x\cdot c) e(-x \cdot d).\]
\[ = \frac{1}{|\mathcal{P}|^4} \sum_{\substack{a,c \in E \\ b,d \in E}} \sum_{x \in  \mathbb{F}^3}e(x \cdot (a+b-c-d))  = \frac{| \mathbb{F}|^3}{|\mathcal{P}|^4} \sum_{\substack{a+b = c+d \\ a,c \in E \\ b,d \in E}} 1.\]
Using the fact that $E \subseteq \mathcal{P}$, we observe:
\[\sum_{\substack{a+b = c+d \\ a,c \in E \\ b,d \in E }} 1 = \sum_{\substack{a-d = c-b \\ a,c \in E \\ b,d \in E }} 1 \leq \sum_{\substack{a-d = c-b \\ a \in E \\ b,d \in E \\ c \in \mathcal{P}}} 1.\]
This can be bounded above by:
\[\leq |E| \cdot \max_{b \in E} \sum_{\substack{a-d = c-b \\ a \in E \\ d \in E \\ c \in \mathcal{P}}} 1 \leq |E| \cdot \max_{ b\in \mathcal{P}} \sum_{\substack{a-d = c-b \\ a \in E \\ d\in E \\ c \in \mathcal{P}}} 1 .\]
We now consider the quantity inside the maximum for an arbitrary, fixed $b \in \mathcal{P}$. To bound this, we will use the Galilean transformation $g_{\delta}: \mathcal{P} \rightarrow \mathcal{P}$, which is defined for each $\delta \in  \mathbb{F}_*^{2}$ by:
\[g_{\delta}(\gamma, \tau) := (\gamma+ \delta, \tau + 2 \gamma \cdot \delta + \delta \cdot \delta),\]
where $(\gamma, \tau) \in  \mathbb{F}_*^{2} \times  \mathbb{F}_* =  \mathbb{F}_*^3$. We note that for each $\delta \in  \mathbb{F}_*^{2}$, this is a bijective map from $\mathcal{P}$ to itself. It now easily follows that (see Claim 5 in \cite{lewko} for a proof):
\begin{claim} We write $b \in \mathcal{P}$ as $b = (\nu, \nu \cdot \nu)$, for $\nu \in F_*^{2}$. We also define $E':= g_{-\nu} (E)$. We then have:
\[\sum_{\substack{a-d = c-b\\ a \in E \\ d \in E \\ c\in \mathcal{P}}} 1 = \sum_{\substack{a'-d' \in \mathcal{P} \\ a' \in E' \\ d' \in E'}} 1.\]
\end{claim}
Thus it suffices to bound the quantity
\[\sum_{\substack{a'-d' \in \mathcal{P} \\ a' \in E' \\ d' \in E'}} 1.\]
We note that the contribution to this sum from terms where $d' = 0$ is at most $|E'| = |E|$. On the other hand, there can be no contribution from terms where $a' = 0$ and $d' \neq 0$, since having both of $d', -d'$ in $\mathcal{P}$ is impossible for $d' \neq 0$. Hence, we have:
\[\sum_{\substack{a'-d' \in \mathcal{P} \\ a' \in E' \\ d' \in E'}} 1 \leq |E| + \sum_{\substack{a'-d' \in \mathcal{P}\\ a' \in E'-\{0\} \\ d' \in E' - \{0\}}} 1.\]
We now define the set $X_{E'} := \{\gamma \in  \mathbb{F}_*^{2} : (\gamma, \gamma \cdot \gamma) \in E' - \{0\}\}$. Letting $a' = (x, x \cdot x)$ and $d' = (y, y \cdot y)$, we note that $a' - d' \in \mathcal{P}$ is equivalent to $x \cdot y = y\cdot y$.
For each $y \in  \mathbb{F}_*^{2}$, we can define a line in $ \mathbb{F}_*^{2}$ by $\ell(y) := \{x \in  \mathbb{F}_*^{2}: y \cdot x = y \cdot y\}$. We now prove that these lines are distinct, i.e. $y$ and $\ell(y)$ are in bijective correspondence. We suppose that $\ell(y) = \ell(y')$ for $y, y' \in  \mathbb{F}_*^2$. We note that $y \in \ell(y)$ and $y' \in \ell(y')$. Since these lines are the same, we must also have $y \in \ell(y')$ and $y' \in \ell(y)$. By definition of $\ell(y), \ell(y')$, this implies that $y \cdot y = y' \cdot y = y' \cdot y'$. Hence,
$(y-y') \cdot (y- y')  = y \cdot y - 2 y' \cdot y + y' \cdot y' = 0$.
However, since $-1$ is not a square in $\mathbb{F}$, this implies that $y - y'$ must be the zero vector in $F_*^2$. Thus, $y = y'$.
We define $L_{E'}$ to be the collection of lines $L_{E'} := \{\ell(y): y \in X_{E'}\}$. Since these lines are distinct and $a'-d' \in \mathcal{P}$ if and only if the corresponding $x,y$ satisfy $x \in \ell(y)$, we have that:
\[\sum_{\substack{a'-d' \in \mathcal{P}\\ a' \in E'-\{0\} \\ d' \in E' - \{0\}}} 1 = \left|\{(\ell(y), x) \in L_{E'} \times X_{E'}: x \in \ell(y)\}\right|.\]
\[=|I(E', L_{E'})|  \]
Thus
\[ \sum_{\substack{a+b = c+d \\ a,c \in E \\ b,d \in E}} 1 \leq |E| \left(|E| + I(E', L_{E'}) \right) \]
and recalling that $\left|\left|(\chi_E d\sigma)^\vee \right|\right|_{L^4( \mathbb{F}^3,dx)}^4 = \frac{| \mathbb{F}|^3}{|\mathcal{P}|^4} \sum_{\substack{a+b = c+d \\ a,c \in E \\ b,d \in E}} 1$ we see that
\[\left|\left|(\chi_E d\sigma)^\vee \right|\right|_{L^4( \mathbb{F}^3,dx)}^4 \leq 2 \cdot \frac{|E||\mathbb{F}|^{3}}{|\mathcal{P}|^4} \left(|E| + I(E', L_{E'}) \right).\]
Finally, from the incidence hypothesis we have
\[\left|\left|(\chi_E d\sigma)^\vee \right|\right|_{L^4( \mathbb{F}^3,dx)} \ll  \left(\frac{|\mathbb{F}|^3}{|\mathbb{F}|^8} |E|^{1+\alpha} \right)^{1/4}.\]
which completes the proof.
\end{proof}
\section{Reduction to regular functions}
First we introduce the notion of a regular set and a regular function. First we define a regular set, $A \subseteq \mathbb{F}^3$. Let $0 \leq \gamma \leq 3$ such that $|A| = |\mathbb{F}|^{\gamma}$. We define $A_{z} \subseteq \mathbb{F}^2$ to be the restriction of $A$ to the hyperplane $\{(x_1,x_2,x_3): x_3=z\}$. We say $A$ is regular (or, when we wish to be more precise, $(\gamma, s, t)$-regular) if the following holds. We have $\gamma = s+t$ where $t$ is defined to be the number $|\{z \in \mathbb{F} : |A_{z}| > 0 \}| = |\mathbb{F}|^t$ and if $|A_z| \geq 0$ then $|\mathbb{F}|^{s}  \leq |A_{z}| \leq 2 |\mathbb{F}|^{s}$. Furthermore, we will call $g : \mathbb{F}^3 \rightarrow \mathbb{C}$ a $(\gamma, s, t)$ -regular function if $g$ is supported on a $(\gamma, s, t)$-regular set, and $1/2 \leq |g| \leq 1$ on its support.
Recall our goal is to prove an inequality of the form
\[ ||\widehat{f}||_{L^{p'}(S,d\sigma)} \ll ||f||_{L^{q'}(\mathbb{F}^n,dx)}.\]
The first part of the argument is the familiar dyadic pigeonholing to level sets. More precisely, we fix $||f||_{L^{q'}(\mathbb{F}^n,dx)}=1$ and decompose $f = \sum_{ 0 \leq i \leq 10 \log(|F|)} 1_{S_i} f + E$ where $S_i := \{x \in \mathbb{F}^3 : 2^{i-1} < |f(x)| \leq 2^{i} \}$ and $E = f$ if $|f| \ll |\mathbb{F}|^{-10}$, and $0$ otherwise. The contribution from $E$ is easily seen to be neglige by H\"{o}lder's inequality (or Lemma \ref{Lem:STinfty}). Since there are only $O(\log(|F|))$ terms in the sum, if we show that (uniformly) $||\widehat{f 1_{S_i}}||_{L^{2}(S,d\sigma)} \ll \log^{O(1)}(|\mathbb{F}|) ||f 1_{S_i}||_{L^{q'}(\mathbb{F}^n,dx)}$ for all $i$, then it follows that $||\hat{f}||_{L^{p'}(S,d\sigma)} \ll \log^{O(1)}(|\mathbb{F}|) || f||_{L^{q'}(\mathbb{F}^n,dx)}  $ and we may then apply the $\epsilon$ removal lemma (Lemma \ref{lem:eprem}) to conclude that $\mathcal{R}^{*}(2\rightarrow q+\epsilon)$ for any $\epsilon >0$. Moreover, by scaling it suffices to show that $||\widehat{f 1_{S_i}}||_{L^{2}(S,d\sigma)} \ll \log^{O(1)}(|\mathbb{F}|) || 1_{S_i}||_{L^{q'}(\mathbb{F}^n,dx)}$ with $ 1/2 \leq  |f| \leq 1$.
We now claim that it suffices to only consider regular functions. To see this let $g=f1_{S}$ for some $ 1/2 \leq  |f| \leq 1$. For $z \in \mathbb{F}$ we define $g_{z}:= g(x_1,x_2,z)$ to be the restriction of $g$ to the hyperplane $\{(x_1,x_2,x_3): x_3=z\}$.  We let $A_{z} \subseteq \mathbb{F}^2$ denote the support of $g_z$.  Clearly, $0 \leq |A_{z}| \leq |\mathbb{F}|^2$.  For each $0\leq  i \leq 10\log(|\mathbb{F}|)$ we may partition $\mathbb{F}$ (the domain of the $z$ parameter) into sets $J_i$ such that $2^{i} \leq  |A_{z}| < 2^{i+1}$ for all $z \in J_i$. We now define $B_i = \cup_{z \in J_i} A_{z}$.  This gives a decomposition of $g = \sum_{i=1}^{10\log(|\mathbb{F}|)}  g 1_{B_i} $ where each term, $g 1_{B_i}$, has disjoint support and is a regular function. Now if we have
\[ ||\widehat{g 1_{B_i}}||_{L^{p'}(S,d\sigma)} \ll \log^{O(1)}(|\mathbb{F}|)||1_{B_i}||_{L^{q'}(\mathbb{F}^n,dx)}\]
for each $B_{i}$, then clearly
\[ ||\widehat{g 1_{B_i}}||_{L^{p'}(S,d\sigma)}  \ll \log^{O(1)}(|\mathbb{F}|)  \max_{i} ||1_{B_i}||_{L^{q'}(\mathbb{F}^n,dx)} \ll  \log^{O(1)}(|\mathbb{F}|)|| f 1_{S} ||_{L^{q'}(\mathbb{F}^n,dx)}  \]
which, as discussed above, is sufficient for our purposes.  We have thus proved that
\begin{lemma}\label{lem:regular}If the inequality
\[||\widehat{f}||_{L^{p'}(S,d\sigma)} \ll \log^{O(1)}(|\mathbb{F}|) ||f||_{L^{q'}(\mathbb{F}^n,dx)}\]
holds for all regular functions $f$, then
\[||\widehat{f}||_{L^{p'}(S,d\sigma)} \ll_{\epsilon} ||f||_{L^{q' -\epsilon}(\mathbb{F}^n,dx)}  \]
holds for all functions and every $\epsilon >0$.
\end{lemma}
\section{Proof of Theorem \ref{thm:incTores}}
We start by proving some estimates for regular functions.
\begin{lemma}\label{lem:mt1} Let $h$ be a $(\gamma, s, t)$-regular function. Furthermore, assume that for any set $E \subseteq \mathbb{F}^2$ of size $\frac{1}{2}|\mathbb{F}|^s \leq |E| \leq 2 |\mathbb{F}|^s$, we have that $||(1_{E} d\sigma)^\vee ||_{L^4(\mathbb{F}^3, dx)} \leq |E|^{(1+\alpha)/4}|F|^{-5/4}$. Then
$$||\hat{h}||_{L^{2}(S,d\sigma)} \ll   || h ||_{L^{2}(\mathbb{F}^3,dx) } + || h ||_{L^{\frac{8(s+t)}{7t-1+s(4+\alpha)}}(\mathbb{F}^3,dx) }.$$
\end{lemma}
\begin{proof}
Define $h_{z}(x_1,x_2,x_3) = h(x_1,x_2,x_3)$ for $x_3=z$ and $0$ otherwise. Thus $h = \sum_{z} h_{z}$. We start by estimating
$$||h_{z} * (d\sigma)^{\vee}||_{L^{4}(\mathbb{F}^3,dx)}.$$
Note by translation symmetry, we may assume that $z=0$. Recall that $(d\sigma)^{\vee} = \delta_{0} + K$ where $K$ is the Bochner-Riesz kernel associated to $\mathcal{P}$. From its definition we have that $K(\underline{x},x_3) = |\mathbb{F}|^{-2} S(x_3)^2 e(\underline{x}\cdot \underline{x}/4 x_3)$ where $S(x_3 ) = \sum_{\xi \in \mathbb{F}} e(x \xi^2)$ is a Gauss sum. Thus
\[||h_{z} *K ||_{L^{4}(\mathbb{F}^3,dx)} =\left( \sum_{x_3 \in \mathbb{F}, x_3 \neq 0}  \sum_{\underline{x} \in \mathbb{F}^2} \left| |\mathbb{F}|^{-2} |S(x_3)|^2 \sum_{y \in \mathbb{F}^2} h_{0}(\underline{y},0) e( (\underline{x}-\underline{y})\cdot(\underline{x}-\underline{y})/4x_3)    \right|^4  \right). \]
Using the pseudo-conformal transformation (for $x_3 \neq 0$) $t:=1/4x_3$ and $z:= -\underline{x}/2x_3$, we have that
\[ (\underline{x}-\underline{y}) \cdot (\underline{x}-\underline{y}) / 4 x_3 = z^2 x_3 +x \cdot \underline{y} + t \underline{y}^2. \]
Additionally noting that $|S(x_3)|^2 = |\mathbb{F}|$ (for $x_3 \neq 0$) by standard Gauss sum estimates (see \cite{MT}, for instance). We now have that
\[|| h_{z} * K||_{L^{4}(\mathbb{F}^3,dx)} \ll |\mathbb{F}| \left( \sum_{(z,t) \in \mathbb{F}^3} |(h_z d\sigma)^{\vee}(z,t)|^{4}  \right)^{1/4}  \]
\[ \ll |\mathbb{F}|^{-1/4+s(1+\alpha)/4}.\]
where we have used the hypothesis that $||(1_{E} d\sigma)^\vee ||_{L^4(\mathbb{F}^3, dx)} \leq |E|^{(1+\alpha)/4}|F|^{-5/4}$. Thus,
\[ || h * K||_{L^{4}(\mathbb{F}^3,dx)} \leq \sum_{z}|| h_{z} * K||_{L^{4}(\mathbb{F}^3,dx)} \ll |\mathbb{F}|^{t-1/4+s(1+\alpha)/4}.\]
Since $||h * \delta_{0}||_{L^4(\mathbb{F}^3, dx)} \ll ||h ||_{L^4(\mathbb{F}^3, dx)} \ll |\mathbb{F}|^{(s+t)/4} $ by Young's inequality, it follows that
\[ || h * (d \sigma)^{\vee}||_{L^{4}(\mathbb{F}^3,dx)} \ll |\mathbb{F}|^{t-1/4+s(1+\alpha)/4} + |\mathbb{F}|^{(s+t)/4}. \]
By H\"{o}lder's inequality this implies
\[ |\left<h, h * (d \sigma)^{\vee} \right> | \ll  \left( |\mathbb{F}|^{t-1/4+s(1+\alpha)/4}  + |\mathbb{F}|^{(s+t)/4}\right) ||h||_{L^{4/3}(\mathbb{F}^3,dx)} .\]
which gives
\[ |\left<h, h * (d \sigma)^{\vee} \right> | \ll  \left( |\mathbb{F}|^{t-1/4+s(1+\alpha)/4}  + |\mathbb{F}|^{(s+t)/4}\right) |\mathbb{F}|^{3(s+t)/4}\]
or
\[ ||\hat{h}||_{L^{2}(S,d\sigma)} \ll  |\mathbb{F}|^{(s+t)/2} +|\mathbb{F}|^{7t/8-1/8+s(1/2+\alpha/8)} . \]
We want to write this last expression in the form $|E|^{1/r} = |\mathbb{F}|^{\frac{s+t}{r}} $. In other words, $r$ is given by $7t/8-1/8+s(1/2+\alpha/8) =\frac{s+t}{r}$. Thus:
\[ ||\hat{h}||_{L^{2}(S,d\sigma)} \ll || h ||_{L^{2}(\mathbb{F}^3,dx) } + || h ||_{L^{\frac{8(s+t)}{7t-1+s(4+\alpha)}}(\mathbb{F}^3,dx) } . \]
\end{proof}
Clearly the estimate
\begin{equation}\label{eq:l2est}
 ||\hat{h}||_{L^{2}(S,d\sigma)} \ll  || h ||_{L^{\frac{8(s+t)}{7t-1+s(4+\alpha)}}(\mathbb{F}^3,dx) } + || h ||_{L^{2}(\mathbb{F}^3,dx) }.
\end{equation}
gets better as the exponent $\frac{8(s+t)}{7t-1+s(4+\alpha)}$ increases. Fixing $s+t$ (and recalling that $0\leq t \leq 1$), and assuming $4/3 \leq \alpha \leq 3/2$ we see that the exponent $\frac{8(s+t)}{7t-1+s(4+\alpha)}$ is smallest (that is least favorable) at $t=1$. Generally we'll set $t=1$ in the following to capture the potential worse case behavior.  However, we note that the hypothesis of Proposition \ref{prop:L4} requires the hypothesis $\mathcal{I}(\alpha, s)$ or $\mathcal{I}(\alpha, \gamma -t)$. As we decrease $t$ we see that the hypothesis $\mathcal{I}(\alpha, \gamma -t)$ becomes stronger. Thus we can only reason that $t=1$ is the worst case if $\mathcal{I}(\alpha, \gamma -t)$ holds for all $0 \leq t\leq 1$. However, we also always have the trivial hypothesis $\mathcal{I}(3/2, \gamma -t)$, regardless of $\gamma$ and $t$. Thus, to reduce to the case when $t=1$ it suffices to only require $\mathcal{I}(\alpha, \gamma -t)$ for values of $t$ bigger than those for which assuming $\mathcal{I}(3/2, \gamma -t)$ would provide a sufficient estimate. More precisely, we have
\begin{lemma}\label{lem:reg}Let $h$ be a $(\gamma, s, t)$-regular function. If $\mathcal{I}(\alpha, \frac{2}{3}(3\gamma- \alpha \gamma - 3 + \alpha ))$ holds then
\[ ||\hat{h}||_{L^{2}(S,d\sigma)} \ll  || h ||_{L^{\frac{8\gamma}{6+(\gamma - 1)(4+\alpha)}}(\mathbb{F}^3,dx) } + || h ||_{L^{2}(\mathbb{F}^3,dx) }. \]
\end{lemma}
\begin{proof}If $\mathcal{I}(\alpha, \gamma)$ holds then the conclusion is just Lemma \ref{lem:mt1} with $t=1$, by the previous discussion.  Now, for fixed $\gamma$, lets assume we have $\mathcal{I}(\alpha, \beta)$ for some $\alpha< 3/2$ and $\beta \leq \gamma$, as well as the trivial estimate $\mathcal{I}(3/2, \gamma)$. We set the implied exponents in (\ref{eq:l2est}) equal and compute how small $t$ must be before the estimate implied by the trivial bound is superior.
\[\frac{8\gamma }{6+(\gamma-1)(4+\alpha)} = \frac{8\gamma }{7t-1+(\gamma -t)(4+3/2)}  \]
which gives
\[ 6+(\gamma-1)(4+\alpha) = 7t-1+(\gamma -t)(11/2) \]
Solving for $t$ we have
\[t = \frac{2}{3}\alpha(\gamma-1)-\gamma +2.\]
We may conclude that whenever $t \leq \frac{2}{3}\alpha(\gamma-1)-\gamma +2$ we have
$$ ||\hat{h}||_{L^{2}(S,d\sigma)} \ll  || h ||_{L^{\frac{8\gamma}{6+(\gamma - 1)(4+\alpha)}}(\mathbb{F}^3,dx) }.$$
 Thus we only require the hypothesis $\mathcal{I}(\alpha, \beta)$ on sets of size smaller than $\gamma -t =  \frac{2}{3}(3\gamma- \alpha \gamma - 3 + \alpha )$, which is the claim.
\end{proof}
\begin{remark}\label{rem:alpha1}As one can see from the above argument, one only needs $\mathcal{I}(\alpha, \gamma-1)$ to hold when $t=1$ and a family of weaker claims such as $\mathcal{I}(\alpha(t), \gamma -t)$ where $\alpha(t)$ increase as $t$ decreases to carry out the above argument.
\end{remark}
Thus assuming $\mathcal{I}(\alpha, \frac{2}{3}(3\gamma- \alpha \gamma - 3 + \alpha ))$, we have the estimate:
\[||\hat{h}||_{L^{2}(S,d\sigma)} \ll  || h ||_{L^{\frac{8\gamma}{6+(\gamma - 1)(4+\alpha)}}(\mathbb{F}^3,dx) }+ || h ||_{L^{2}(\mathbb{F}^3,dx) } \]
which, for fixed $4/3 \leq \alpha \leq 3/2 $, improves as $\gamma$ increases (since $\frac{d}{d\gamma} \frac{8\gamma}{6+(\gamma-1)(4+\alpha)}=\frac{8(2-\alpha)}{(\alpha(\gamma-1)+4\gamma+2)^2}$).  On the other hand, by Corollary \ref{cor:stdecay} we have, for all sets $E$ with $|E| = |\mathbb{F}|^{\gamma}$, that
\begin{equation}
 ||\widehat{1_{E}}||_{L^{2}(S,d\sigma)} \ll  ||1_{E} ||_{L^{\frac{2\gamma}{2\gamma-1}}(\mathbb{F}^n, dx)} + || 1_{E}  ||_{L^{2}(\mathbb{F}^3,dx) }
\end{equation}
This estimate improves as $\gamma$ decreases. Setting the exponents in our two estimates equal we have
\[\frac{2\gamma}{2\gamma-1} = \frac{8\gamma}{6+(\gamma - 1)(4+\alpha)}. \]
Solving for $\gamma$ we have
\[\gamma = \frac{6-\alpha}{4-\alpha }.\]
Thus by applying Lemma \ref{cor:stdecay} when $ 0 \leq \gamma \leq \frac{6-\alpha}{4-\alpha }$ and Lemma \ref{lem:reg} when $\frac{6-\alpha}{4-\alpha } \leq \gamma \leq 3$ we can conclude (for $\frac{12-2\alpha}{8-\alpha} \leq 2$) that
\[ ||\widehat{1_{E}}||_{L^{2}(S,d\sigma)} \ll || 1_{E} ||_{L^{ \frac{12-2\alpha}{8-\alpha}} (\mathbb{F}^n, dx)}  \]
if $\mathcal{I}(\alpha, \frac{2}{3}( \gamma(3-\alpha) - 3 + \alpha )$ holds for $\gamma \geq  \frac{12-2\alpha}{8-\alpha}$. Of course, this won't hold for very large $\gamma$ if $\alpha < \frac{3}{2}$. On the other hand, for large enough $\gamma$ the estimate implied by the trivial incidence theorem $\mathcal{I}(\frac{3}{2},\gamma)$, will be sufficient. This will occur when $\gamma$ satisfies
\[\frac{12-2\alpha}{8-\alpha}  = \frac{8\gamma}{6+(\gamma - 1)(11/2)}. \]
Solving this gives $\gamma = \frac{6-\alpha}{3\alpha -2}$. Thus to have the conclusion of Theorem \ref{thm:incTores} we need  $\mathcal{I}(\alpha, \frac{2}{3}(3\gamma- \alpha \gamma - 3 + \alpha ))$ to hold only for $\frac{6-\alpha}{4-\alpha } \leq \gamma \leq \frac{6-\alpha}{3\alpha -2}$. Since $\frac{2}{3}(3\gamma- \alpha \gamma - 3 + \alpha )$ is an increasing function of $\gamma$ we see that it suffices to require  $\mathcal{I}(\alpha, \frac{2}{3}(3\gamma- \alpha \gamma - 3 + \alpha ))$ for $\gamma = \frac{6-\alpha}{3\alpha -2}$. That is:
\[ I\left( \alpha, \frac{8(\alpha -3)(\alpha -2) }{9\alpha -6  }\right).\]
We thus conclude that if $\mathcal{I}\left( \alpha, \frac{8(\alpha -3)(\alpha -2) }{9\alpha -6  }\right)$ we have
\[   ||\widehat{f}||_{L^{2}(S,d\sigma)} \ll || f||_{L^{ \frac{12-2\alpha}{8-\alpha}} (\mathbb{F}^n, dx)}   \]
for all regular functions $f$. By Lemma \ref{lem:regular}, we may now conclude:
\newtheorem*{thm:incTores}{Theorem \ref{thm:incTores}}
\begin{thm:incTores}
Let $\mathbb{F}$ be a finite field such that $-1$ is not a square. Furthermore, assume $I\left( \alpha, \frac{8(\alpha -3)(\alpha -2) }{9\alpha -6  }\right)$. Then
\[\mathcal{R}^{*}\left(2\rightarrow \frac{12 - 2 \alpha}{4-\alpha} + \epsilon \right) \ll 1  \]
for all $\epsilon >0$.
\end{thm:incTores}
\begin{remark}\label{rem:alpha2}It is clear from the above argument that we only need the full strength of the of the estimate $\mathcal{I}(\alpha, \beta)$ at $\gamma = \frac{6-\alpha}{4-\alpha}$, and then progressively weaker values of $\alpha$ as $\gamma$ increases to $\frac{6-\alpha}{3\alpha-2}$. In turn, by the previous remark we see that at $\gamma = \frac{6-\alpha}{4-\alpha}$ we only need the full strength of $\mathcal{I}(\alpha, \beta)$ at $\beta = \gamma-1 = \frac{2}{4-\alpha}$, and then progressively weaker values of $\alpha$ as $\beta$ increases to $\frac{8(\alpha -3)(\alpha -2) }{9\alpha -6  }$.
\end{remark}
\section{General finite fields with $-1$ not a square}
We will now consider the case in which the finite field $\mathbb{F}$ might not be of prime order. Essentially none of the previous discussion explicitly required the finite field to be of prime order, however in general finite fields statements such $\mathcal{I}(\alpha, \gamma)$ with $\alpha < \frac{3}{2}$ can fail if the field contains a subfield of a certain size. In the appendix we will prove an incidence theorem in general finite fields which will be able to combine with the previous arguments. The argument closely follows the original argument of Bourgain, Tao and Katz \cite{BKT}. Our goal now is to prove the following:
\newtheorem*{thm:general}{Theorem \ref{thm:general}}
\begin{thm:general}
Let $\mathbb{F}$ be an arbitrary finite field such that $-1$ is not a square.  Then
\[\mathcal{R}^{*}\left(2\rightarrow \frac{18}{5}-\epsilon \right) \ll 1  \]
holds for some $\epsilon >0$.
\end{thm:general}
First, we describe a few adjustments to the setup used in the previous section.   We let $0 < \delta < \frac{1}{1000}$ be a small parameter to be set later. We set the problem up as before, applying Lemma \ref{lem:regular} to reduce to the case of regular functions. We let $f$ be a $(\gamma, s,t )$-regular function supported on the $(\gamma, s,t )$-regular set $E$. If $\gamma \leq \frac{9}{5} - \delta$ then by Corollary \ref{cor:stdecay}, we have that
\[ ||\widehat{f }||_{L^{2}(S,d\sigma)} \ll ||f||_{L^{ \frac{18}{13} + \delta'  } (\mathbb{F}^n, dx)}.\]
Similarly, by Lemma \ref{lem:charinfST} and Proposition \ref{prop:localres} if $\gamma \geq \frac{9}{5} + \delta$ we have
\[ ||\widehat{f }||_{L^{2}(S,d\sigma)} \ll ||f||_{L^{ \frac{18}{13} + \delta''  } (\mathbb{F}^n, dx)}.\]
Thus it suffices to consider regular functions with $\ \frac{9}{5} - \delta \leq \gamma \leq  \frac{9}{5} + \delta$.
Now let $\frac{\theta}{100} = \delta >0$, $t \geq 1-\theta$  and assume that $s+t =\gamma \in ( 9/5 - \delta, 9/5 +\delta)$.
By Lemma \ref{lem:mt1}, we have
\[||\hat{f}||_{L^{2}(S,d\sigma)} \ll || f ||_{L^{\frac{8(s+t)}{7t-1+s(11/2)}}(\mathbb{F}^3,dx)} \]
just from the trivial incidence bound at $\alpha=3/2$. Now we claim that if $\delta >0$ is sufficiently small this will be $\ll ||f||_{L^{ \frac{18}{13} + \delta'''  } (\mathbb{F}^n, dx)} $ for some small $\delta''' >0$  in the range of parameters specified. To see this we note that the exponent on the right is lower bounded by
\[\frac{8\cdot(9/5 - \delta )}{7(1-\theta) -1+(4/5+\delta+\theta)(11/2)} \geq \frac{ 72 - 2\theta/5 }{52 - 289\theta/40} \geq  \frac{18}{13} + \frac{74\theta}{ 6760 - 91 \theta}  \geq  \frac{18}{13} + \delta'''. \]
Thus we have reduced to the case of $(\gamma,s,t)$-regular functions, where $\gamma \in ( 9/5 - \delta, 9/5 +\delta)$ and $t>1-\theta$.  We can then conclude the proof by an application of Lemma \ref{lem:mt1} if we can improve the trivial incidence bound for all sets of dimension $4/5 -\epsilon \leq  s \leq 4/5 + \epsilon$ for some $\epsilon>0$.  Roughly speaking, it turns out that the only case in which one does not have a non-trivial incidence estimate is when (an appropriate transformation of) the set $E$ is contained in the Cartesian product $G \times G$ of a subfield $G \subseteq \mathbb{F}$ with itself. However, we can never have a subfield of $\mathbb{F}$ of dimension near $2/5$ since if $|\mathbb{F}| = p^\ell$, then $|G|=p^{j}$ for some integer $j \mid \ell$.  More precisely, we have that:
\newtheorem*{thm:inc}{Theorem \ref{thm:inc}}
\begin{thm:inc}
Let $P \subseteq \mathbb{F}^2$ of a finite field $\mathbb{F}$, and $L$ a set of lines in $\mathbb{F}^2$ such that $|P|=|L|=N$. Moreover, assume that
\[ |I(P,L)| \geq |\mathbb{F}|^{-\epsilon} N^{3/2}.\]
Then, for some large absolute constant $C$, there exists the following: (i) a subfield $G \subseteq F$ such that $|G| \leq C |\mathbb{F}|^{\epsilon C}|P|^{1/2}$ (ii) sets of the form $S_{1}=a_{1}+b_{1}\cdot G$ and $S_{2}=a_{2}+b_{2}\cdot G$, (iii) a projective transformation\footnote{See the discussion in the appendix for a definition of this.} $T$, such that
\[    | \left(S_{1} \times S_{2}\right) \cap T(P) | \geq C^{-1} |\mathbb{F}|^{-\epsilon C} |P| .\]
\end{thm:inc}
Here we've taken $K=|\mathbb{F}|^{\epsilon}$ (and noted that we may take $G_1 = G_2$) in the formulation below. Essentially, this follows from the arguments of Bourgain, Katz and Tao \cite{BKT}, using the sum-product theorem in general finite fields from \cite{TV}. Lacking a reference to such a formulation in the literature, we have included a proof in the appendix.
\appendix
\setcounter{secnumdepth}{0}
\section{Appendix: A General Incidence Theorem}
This appendix is devoted to proving a variant of the Bourgain-Katz-Tao incidence theorem. While this is a statement about incidences in the finite plane $\mathbb{F}^2$, to properly state the result we need to introduce projective space $\mathbb{P}\mathbb{F}^3$. Our description will be somewhat brief, and we refer the reader to, for instance, Dvir's survey article \cite{Dvir} for a more thorough description in this context. Recall that $\mathbb{P}\mathbb{F}^3$ is the set of points in $\mathbb{F}^3 \setminus (0,0,0)$ modulo dilations. We can embed $\mathbb{F}^2$ into this space by mapping $(x,y) \rightarrow (x,y,1)$. Thus $\mathbb{P}\mathbb{F}^3$ is equivalent to $\mathbb{F}^2$ union the `line at infinity'.  Now the mapping of $\mathbb{F}^2$ into $\mathbb{P}\mathbb{F}^3$ maps  points to points, lines to lines, and preserves incidences. By a projective transform we will mean a $3\times 3$ invertible matrix acting in the natural way on $\mathbb{P}\mathbb{F}^2 =  \mathbb{F}^3 \setminus (0,0,0)$. We note that this also maps points to points, lines to lines and preserves incidences.
To allow us to state the result in terms of $\mathbb{F}^2$ instead of $\mathbb{P}\mathbb{F}^3$ we will define a projective transformation of a subset $A \subseteq \mathbb{F}^2$ as follows. We let $T$ be a projective transform on $\mathbb{P}\mathbb{F}^3$ and we embed $A$ into $\mathbb{P}\mathbb{F}^2$ using the natural mapping $(x,y) \rightarrow (x,y,1)$. We then apply $T$ to this set and then map back into $\mathbb{F}^2$ using the mapping $(x,y,1) \rightarrow (x,y)$. The unfortunate complication, is that it is possible that $T$ maps some of the points onto the `line at infinity', and thus `applying a projective transformation to a set $A\subseteq \mathbb{F}^2$' has a number of undesirable properties, such as possibly reducing the cardinality of the set. It turns out that the number of points we lose in this process will always be negligible for our purposes. We now may state the result:
\begin{theorem}\label{thm:inc}Let $\mathbb{F}$ be a finite field. Let $P$ be a set of points and $L$ a set of lines in $\mathbb{F}^2$ such that $|P|=|L|=N$. Moreover, assume that
\[ |I(P,L)| \geq K^{-1} N^{3/2}.\]
Then, for a large absolute constant $C$, there exists the following: (i) subfields $G_{1}, G_{2} \subseteq F$ such that $|G_{1}|, |G_{2}| \leq C K^{C} N^{1/2}$ (ii) sets of the form $S_{1}=a_{1}+b_{1}\cdot G_{1}$ and $S_{2}=a_{2}+b_{2}\cdot G_{2}$, (iii) a projective transformation $T$, such that
\[    | \left(S_{1} \times S_{2}\right) \cap T(P) | \geq C^{-1} K^{-C} |P| .\]
\end{theorem}
We follow the expositions in \cite{BKT} and \cite{Dvir}. The existence of a result of this form is alluded to in \cite{BKT}, however no precise formulation is given there. In addition, Jones has given an explicit incidence theorem in general finite fields \cite{Jones2}, however his formulation and parameters aren't sufficient for our purpose.
First we need the following sum-product theorem, as given it is Theorem 2.55 from \cite{TV}. This is a slight strengthening of the results of Bourgain, Katz and Tao \cite{BKT} and Bourgain, Glibichuk and Konyagin \cite{BGK}.
\begin{theorem}\label{lem:sumprod}(Sum-product theorem) Let $A$ be a finite non-empty subset of a field $F$, and let $K \geq 1$. Then the following statements are equivalent in the sense that if the first statement holds with constant $C_1$, then the second holds with some other constant $C_2$ (depending only on $C_{1}$) and vise versa.
(i) $|A+A| \leq C_{1}K^{C_1} |A|$ and $|A\cdot A| \leq C_1 K^{C_1}|A|$
(ii) There exists a subfield $G$ of $\mathbb{F}$, a non-zero element $x \in \mathbb{F}$, and a set $X$ of $\mathbb{F}$ such that $|G| \leq C_{2}K^{C_2}|A|$,  $|X| \leq C_2 K^{C_2}$ and
$A \subseteq x \cdot G \cup X$.
\end{theorem}
Next we will need the following sum-set estimate, which can be found as Lemma 2.2 in \cite{BKT}.
\begin{lemma}\label{lem:sumset}(Sumset theorem) Let $A,B$ be non-empty finite subsets of an additive group such that $|A+B| \leq K \min \left(|A|,|B| \right)$. Then
$$|A \pm A| \leq C K^{C}|A|$$
for some absolute constant $C$.
\end{lemma}
We also need the Balog-Szemer\'{e}di-Gowers lemma. This can be found, for instance, as Theorem 2.29 in \cite{TV}.
\begin{lemma}\label{lem:bsg}(Balog-Szemer\'{e}di-Gowers)Let $A,B$ be finite subsets of an additive group with cardinality $|A|=|B|$, and let $G$ be a subset $A\times B$ with cardinality $|G| \geq |A||B|/K $ such that
$$|\{ a+b : (a,b) \in G \}| \leq K|A|.$$
From some large universal constant $C$, there exists subsets $A'$, $B'$ of $A$ and $B$ respectively with $|A'| \geq C K^{-C}|A|$ and $|B'| \geq C K^{-C}|B|$ such that
$|A' - B'| \leq C K^{C}|A|.$
\end{lemma}
Finally, we recall the trivial incidence inequality:
\begin{equation}\label{eq:trivInc}
\left|I(P,L)\right| \leq \min\left( |P|^{1/2}|L| +|P|, |P||L|^{1/2}  +|L| \right).
\end{equation}
\begin{lemma}\label{lem:algstruct}Let $A,B \subset \mathbb{F}$ such that $|A|,|B| \leq N^{1/2}$ and assume that
\[ | \{y, x_0, x_1) \in B \times A \times A : (1-y) x_0 + yx_1 \in A; y \neq 0,1 \} | \geq K^{-1} N^{3/2}. \]
Then, there exists a subfield $G$ of $\mathbb{F}$, elements $x, \tau \in \mathbb{F}$ ($x \neq 0$), and a set $X \subseteq \mathbb{F}$ such that $|G| \leq CK^{C}|A|$,  $|X| \leq C K^{C}$ and
$A \subseteq (x \cdot G+\tau) \cup X$.
\end{lemma}
\begin{proof}Let $A' \subset A$ be the set of $x_1 \in A$ such that:
\[ | \{(y, x_0) \in B \times A  : (1-y) x_0 + yx_1 \in A; y \neq 0,1 \} | \geq  K^{-3 } N. \]
Since $|A| \leq  N^{1/2}$, we have that
\begin{equation}\label{eq:alginc}
| \{(y, x_0, x_1) \in B \times A \times A' : (1-y) x_0 + yx_1 \in A; y \neq 0,1 \} | \geq C K^{-1} N^{3/2}.
\end{equation}
Moreover,
\[|A'| \geq C K^{-1} N^{3/2} / |A||B| \geq CK^{-1} N^{1/2}.\]
Now by the pigeonhole principle, there exists $y_0 \in B$ ($y_0 \neq 0,1$) such that
\[ | \{(x_0, x_1) \in  A \times A' : (1-y_{0}) x_0 + y_{0}x_1 \in A\} | \geq C K^{- 1} N^{3/2}/|B| \geq C K^{-1} |N|.\]
This implies that
\[ | \{(x_0, x_1) \in  A \times A' : (1-y_{0}) x_0 + y_{0}x_1 \in A\} |  \geq C K^{-1} |A||A'|.\]
Now by Lemma \ref{lem:bsg} we may find a subset $(1-y_0) \tilde{A}$ of $(1-y_0)A$ and $y_0 A''$ of $y_0 A'$ such that $|\tilde{A}|, |A''|\geq C_{1} K^{-C_{1}}N^{1/2}$ and satisfying
\[ |(1-y_0) \tilde{A} - y_0 A' | \leq C_2 N^{1/2}  K^{C_2}.\]
Applying Lemma \ref{lem:sumset} we further may conclude that
$|A' + A'| \leq C_3 K^{C_3} N^{1/2}.$
Again applying the pigeonhole principle to (\ref{eq:alginc}), we see that there must exist $x_0 \in A$ such that
\[| \{ (y,x_1) \in B \times A' : (1-y) x_0 + yx_1 \in A ; y \neq 0,1 \}| \geq C K^{-1} N .\]
By translating $x_0,x_1,A,A'$ we may assume that $x_0=0$. We let $A_{*}', A_{*}$ denote the translations of $A'$ and $A$, respectively and we have that
\[\{ (y,x_1) \in B\setminus \{0\} \times A_{*}'\setminus \{0\} :  yx_1 \in A' ; y \neq 0,1 \}| \geq CK^{-1} N   \]
since that contributions at $0$ are easily seen to be of a lower order. Applying Lemmas \ref{lem:bsg} and \ref{lem:sumset} in multiplicative form, there exists a subset $A'' \subseteq A_{*}'\setminus \{0\}$ with $|A''| \geq C_{4}K^{-C_{4}}N^{1/2}$ such that
\[| A'' \cdot A'' | \leq C_{5}K^{C_5} N^{1/2}.\]
On the other hand $|A'' + A''| \leq C_3 K^{C_3} N^{1/2}$ still holds. Applying the Sum-product Theorem (Theorem \ref{lem:sumprod}) now completes the proof.
\end{proof}
\begin{lemma}\label{lem:gridtoconc}Let $A,B \subset \mathbb{F}$ such that $|A|,|B| \leq N^{1/2}$. Consider the Cartesian product $A \times B \subseteq \mathbb{F}^2$, and let $P \subset A \times B$ and $L$ a set of lines such that $|P|,|L| \leq N$. Then either
\[ |I(P,L)| \leq K^{-1} N^{3/2} \]
or there exists a large constant $C$, subfields $G_{0},G_{1}$ of $\mathbb{F}$, non-zero elements $x_{0}, x_{1}, \tau_{0}, \tau_{1}, \in \mathbb{F}$, and sets $X_{0}, x_{1} \subseteq \mathbb{F}$ such that $|G_{0}|,|G_{1}| \leq C K^{C}|A|$,  $|X_{0}|,|X_{1}| \leq C K^{C}$ and
$A \subseteq (x_{0} \cdot G_{0}+\tau_{0}) \cup X_{0}$, $B \subseteq (x_{1} \cdot G_{1}+\tau_{1}) \cup X_{1}$.
\end{lemma}
\begin{proof}We assume that $|I(P,L)| \geq K^{-1} N^{3/2} $. By symmetry it suffices to prove the conclusion only for the set $A$. Let $L_0 \subseteq L$ be the set of lines which are not horizontal. Since the number of incidences involving the lines in $L \setminus L_{0}$ is at most $N$, we may assume that
\[ | \{(p,l) \in P \times L_{0} : p \in L \}| \geq \frac{1}{2} K^{-1} N^{3/2}, \]
for large $N$.  Furthermore, if we let $L_{1}$ denote the lines in $L_{0}$ such that
\[ | \{(p,l) : p \in \ell \}| \geq C_{2} K^{-2} N^{1/2} \]
then we have that
\[ | \{(p,l) \in P \times L_{1} : p \in L \}| \geq C_{3} K^{-1} N^{3/2} .\]
Now we define $P(b)$ to be the intersection of $P$ with the horizontal line at $b$. That is, $P(b) = P \cap (A \times \{b\} ) $. Since $L_{1}$ contains no horizontal lines we have that
\[ | \{ (b,l) \in B \times L_{1} : P(b) \cap l \neq \emptyset  \}|  = | \{(p,l) \in P \times L_{1} : p \in L \}| \geq C_{3} K^{-1} N^{3/2}. \]
Now by an easy argument involving the Cauchy-Schwarz inequality (see Lemma 5.1 of \cite{BKT} for a detailed proof) it follows, for some large universal constant $C_4$, that
\[ | \{ (b,b', l) \in B \times B \times L_{1} : P(b) \cap l \neq \emptyset, P(b') \cap l \neq \emptyset  \}|  = | \{(p,l) \in P \times L_{1} : p \in L \}| \geq C_{4}^{-1} K^{-C_4} N^{2}. \]
Since $|B| \leq N^{1/2}$, we may now find $b,b' \in B$ ($b \neq b'$) such that
\[| \{ l \in  L_{1} : P(b) \cap l \neq \emptyset, P(b') \cap l \neq \emptyset  \}| \geq C_{5}^{-1} K^{-C_{5}} N   \]
Now we let $\tau(y)=(b'-b)^{-1}(y-b)$ denote an affine transformation of $\mathbb{F}$ and consider the mapping of $B$ under $\tau$, $\tau(B)$. Clearly, $\tau(b)=0$ and $\tau(b')=1$. We define $\tau(L_{1})$ by replacing the line $x=cy+d$ with $x=c\tau^{-1}(y)+d$. Clearly, the new set of lines $\tau(L_{1})$ continues to exclude horizontal lines. Now if we define $\tau(P) \subset A \times \tau(B)$ to be the image of $P$ under the mapping $(x,y)\rightarrow (x, \tau(y))$. We see that this $\tau$ transformation preserves incidences. Now each line in $\tau(L_{1})$ contains $\geq C_{2} K^{-2} N$ points of the form $(x,y)$ in $\tau(P_{0})$ all but at most $2$ of which have $y \neq 0,1$ (since $\tau(L_{1})$ contains no horizontal lines).  Thus
\[ | \{ (x,y,l) \in A \times \tau(B) \times L_{1} : P(0) \cap l \neq \emptyset; P(1) \cap l \neq \emptyset; (x,y) \in l  ; y \neq 0,1  \}|  \geq C_{6}^{-1} K^{-C_6} N^{3/2}. \]
In other words,
\[ | \{ (x,y,l,x_{0},x_{1}) \in A \times \tau(B) \times L_{1} \times A \times A : (x_0,0), (x,y), (x_1,1) \in l ; y \neq 0,1  \}|  \geq C_{6}^{-1} K^{-C_6} N^{3/2}. \]
Note that the points $(x_0,0)$ and $(x_1,1)$ determine $l$ and $(x,y)$ must satisfy
$x= x_0+(x_1 -x_0)y$.
Thus we may rewrite the above as
\[ | \{y, x_0, x_1) \in B \times A \times A : (1-y) x_0 + yx_1 \in A; y \neq 0,1 \} | \geq C_{6} K^{-C_6} N^{3/2}. \]
Now the proof is complete by applying Lemma \ref{lem:algstruct}.
\end{proof}
\begin{lemma}\label{lem:inctogrid}Let $P \subseteq \mathbb{F}^2$ and $L$ a set of lines in $\mathbb{F}^2$ such that $|P|=|L|=N$. Furthermore, assume that
\begin{equation}\label{eq:inHyp}
|I(P,L)| \geq  K^{-1} N^{3/2}.
\end{equation}
Then there exists a large constant $C$, sets $A,B \subseteq \mathbb{F}$ such that $|A|,|B| \leq CK^{C}N^{1/2}$, subsets $P' \subseteq P$ and $L' \subseteq L$ such that $|P'|,|L'| \geq C^{-1} K^{-C} N$, and a projective transformation $T$ such that $T(P') \subseteq A \times B$. Moreover,
\begin{equation}
|I(P',L')| \geq C^{-1} K^{-C} N^{3/2}.
\end{equation}
\end{lemma}
\begin{proof}First we let $L_{1}\subseteq L$ denote the subset of lines each of which contain at least $K^{-2}|N|^{1/2}$ points. Clearly,
\[ |I(P,L_{1})| \geq C_{3}^{-1} K^{-1} N^{3/2}.\]
Moreover we may assume that $|L_{1}| \geq C_{4}^{-1} K^{-1} N  $, since otherwise an application of (\ref{eq:trivInc}) would contradict the hypothesis (\ref{eq:inHyp}). Next, we let $L_{2} \subseteq L_{1}$ be the subset of lines each containing at most $K^2 N^{1/2}$ points. Consider $L^{*} = L_{1} \setminus L_{2}$. Now $|L^{*}| \leq K^{-1}N^{3/2} / K^2 N^{1/2} = K^{-3}N $. Thus, $|L_{2}| \geq C_{5}^{-1} |L_1| \geq C_{5}^{-1} K^{-C_{5}} N  $. Now since each $l \in L_{2} \subseteq L_{1}$ intersects at least $K^{-2}|N|^{1/2}$ points, we have that
\[|I(P, L_{2})| \geq C_{6}^{-1} K^{-C_{6}}N^{3/2} \]
where each $l \in L_{2}$ satisfies $K^{-2}|N|^{1/2} \leq  |l \cap P| \leq K^2 N^{1/2}$. Now an analogous argument allows us to replace $P$ with $P_{1} \subseteq P$, $|P_{1}| \geq C_{7}^{-1} K^{-C_{7}} N^{1/2} $, such that for each $p \in P_{1}$ we have $K^{-C_{8}}|N|^{1/2} \leq |\{\ell \in L_{2} : p \in L\}| \leq   K^{C_{8}} N^{1/2}$. Since this procedure only removes points, clearly each line satisfies $ |l \cap P_{1}| \leq K^2 N^{1/2}$ (however we will no longer claim the lower bound on this intersection).
In summary, we may conclude that there is a constant $\tilde{C}$ such that
\[|I(P_{1}, L_{2})| \geq \tilde{C}^{-1} K^{-\tilde{C}}N^{3/2} \]
for a set of lines $|L_{2}| \geq \tilde{C}^{-1}K^{-\tilde{C}} N$ each of which satisfies $| l \cap P_{1}| \leq \tilde{C} K^{\tilde{C}}  N^{1/2}$, and a set of points $|P_{1}| \geq \tilde{C}^{-1} K^{-\tilde{C}} N$ each of which satisfies $\tilde{C}^{-1}K^{-\tilde{C}} N^{1/2}  \leq   | l \in L_{2} : p \in l | \leq \tilde{C} K^{\tilde{C}} N^{1/2} $.
Now let $L(p) := \{ l \in L_{2} : p \in l \}$ and $P(l) := \{p \in P_{1} : p \in l \}$. We now consider what is known as a bush construction. Let $p \in P_1$ and define
\[U(p) := \{p' \in P_{1}: \exists l \in L_{2} \text{ s.t. } p,p' \in l \}.\]
We now estimate the expected size of $|U(p) \cap U(q)|$ over all points $p,q \in P_{1}$. That is
\[  \mathbb{E}_{p,q} |U(p) \cap U(q)| \geq \frac{1}{N^2} \sum_{p,q \in P_{1}} \sum_{r \in P_{1}} \sum_{l_0, l_2 \in L(r)} 1_{p \in l_0} 1_{q \in \l_1} \]
\[=   \frac{1}{N^2} \sum_{r \in P_{1}} \sum_{l_0, l_2 \in L(r)}|P(l_0)||P(l_1)| =  \frac{1}{N^2} \sum_{r \in P_{1}} \left( \sum_{l_0 \in L(r)}|P(l_0)| \right)^2 \]
\[\geq \frac{1}{N^3}  \left( \sum_{r \in P_{1}} \sum_{l \in L(r)} |P(l)| \right)^2  =  \frac{1}{N^3}  \left( \sum_{l \in L_{2}} |P(l)|^2 \right)^2 \geq   \frac{1}{N^5}  \left( \sum_{l \in L_{2}} |P(l)| \right)^4 \]
\[\geq \frac{1}{N^5}   \left(\tilde{C}^{-1} K^{-\tilde{C}} N^{1/2}    \right)^4   = \tilde{C}^{-4} K^{-4\tilde{C} } N. \]
Thus we can find $p_{0},q_{0}$ satisfying $|U(p_{0}) \cap U(q_{0})| \geq \tilde{C}^{-4} K^{-4\tilde{C} } N$. We let $W = U(p_{0}) \cap U(q_{0})$ and embed this set into projective space $\mathbb{P}\mathbb{F}^2$ with the mapping $(x,y) \rightarrow (x,y,1)$. We may now find a projective transformation which maps the points in $\mathbb{P}\mathbb{F}^2$ corresponding to $p_{0}$ and $q_{0}$ to (the equivalence classes represented by) $(1,0,0)$ and $(0,1,0)$, respectively. Now it is possible that we have moved some of the points to the line at infinity. However, this can only happen if all of these points lie on a single line passing through $p_{0}$ and $q_{0}$, which would then have to be in $L_{2}$. However, we then see that there are at most $\leq \tilde{C} K^{\tilde{C} }  N^{1/2}$ such points which can be readily disregarded. Ignoring these points, we may pass back to $\mathbb{F}^2$ via the map $ (x,y,1) \rightarrow (x,y)$. We now see that each of the lines (obtained from those) passing through $p_{0}$ are horizontal and each of the lines (obtained from those) passing through $q_{0}$ are vertical. Thus we have a projective transformation $T$, a large constant $C_{2}$, and a cartesian grid of the form $A \times B$, where $|A|,|B|\leq C_{2 }K^{C_{2}} N^{1/2}$ such that $\tilde{P} := T(W) \cap A \times B $ satisfies  $|\tilde{P}| \geq C_{2}^{-1} K^{-C_{2}} N^{1/2} $.  Moreover, for each $p \in W$ we have that $|L(p)| \geq  K^{-C_{8}}|N|^{1/2} $. Thus if we let $\tilde{L} = T(L_{2})$ we have that
\[ |I(\tilde{P},\tilde{L})| \geq  C_{2}^{-1} K^{-C_{2}}N \times K^{-C_{8}}|N|^{1/2} \geq C^{-1}K^{-C} N^{3/2}.\]
Taking $P'$ and $L'$ to be the preimage of $\tilde{P}$ and $\tilde{L}$, respectively, under the projective mapping $T$ (and, perhaps, worsening some of the constants) completes the proof.
\end{proof}
We now deduce Theorem \ref{thm:inc} from Lemma \ref{lem:gridtoconc} and Lemma \ref{lem:inctogrid}. We start with a set of points and lines with $|P|,|L| =N$ and satisfying $|I(P,L)| \geq K^{-1}N^{3/2}$. Let us define $M:=C^2K^{2C}N$ and $J:=C^{4}K^{4C}$.  Applying Lemma \ref{lem:inctogrid} and taking $L_{*}=T(L')$ and $P_{*}=T(P')$, we have sets $A$ and $B$ such that $|A|, |B| \leq M^{1/2}$ with $P_{*} \subseteq A \times B$ and satisfying $|I(P_{*},L_{*})| \geq J^{-1} M^{3/2}$. Now applying Lemma \ref{lem:gridtoconc} completes the proof.

\section{Acknowledgements}
We would like to thank Kevin Hughes, Terence Tao, and Josh Zahl for various discussions of related topics.

\texttt{M. Lewko, Department of Mathematics, The University of California, Los Angeles}

\textit{mlewko@gmail.com}


\begin{thebibliography}{5}

\bibitem{BKT}
J. Bourgain, N. Katz, T. Tao, A sum-product estimate in finite fields, and applications. Geom. Funct. Anal. 14 (2004), no. 1, 27--57.

\bibitem{BGK}
J. Bourgain, A. Glibichuk and S. Konyagin, Estimates for the number of sums and products and for exponential sums in fields of prime order. J. London Math. Soc. (2) 73 (2006), no. 2, 380--398

\bibitem{Dvir}
Z. Dvir, Incidence Theorems and Their Applications. arXiv:1208.5073.

\bibitem{Dvir2}
Z. Dvir, On the size of Kakeya sets in finite fields. J. Amer. Math. Soc. 22 (2009), no. 4, 1093--1097.

\bibitem{EOT}
J. Ellenberg, R. Oberlin, and T. Tao, The Kakeya set and maximal conjectures for algebraic varieties over finite fields. Mathematika 56 (2010), no. 1, 1--25.

\bibitem{HR}
H. Helfgott and M. Rudnev, An explicit incidence theorem in $\mathbb{F}_{p}$. Mathematika 57(1):135--145, 2011.

\bibitem{IK}
A. Iosevich and D. Koh, Extension theorems for paraboloids in the finite field setting. Math. Z. 266, (2010) no.2 471--487.

\bibitem{IKs}
A. Iosevich and D. Koh, Extension theorems for spheres in the finite field setting. Forum Math. 22 (2010), no. 3, 457--483.

\bibitem{IKq}
A. Iosevich and D. Koh, Extension theorems for the Fourier transform associated with non-degenerate quadratic surfaces in vector spaces over finite fields. Illinois J. Math. 52, no. 2 (2009), 611--628.

\bibitem{Jones1}
T. Jones, An improved incidence bound for fields of prime order. arXiv:1110.4752

\bibitem{Jones2}
T. Jones, Explicit incidence bounds over general finite fields. Acta Arith. 150 (2011), 241--262.

\bibitem{Jones3}
T. Jones, Further improvements to incidence and Beck-type bounds over prime finite fields. arXix:1206.4517

\bibitem{Koh1}
D. Koh, Restriction operators acting on radial functions on vector spaces over finite fields. arXiv:1212.5298

\bibitem{lewko}
A. Lewko and M. Lewko, Endpoint restriction estimates for the paraboloid over finite fields. Proc. Amer. Math. Soc. 140 (2012) 2013--2028.

\bibitem{lewkoKakeya}
M. Lewko, Finite field restriction estimates based on Kakeya maximal operator estimates. arXiv:1401.8011.

\bibitem{MT}
G. Mockenhaupt and T. Tao, Restriction and Kakeya phenomena for finite fields. Duke Math. J. 121 (2004), no. 1, 35--74.

\bibitem{TV}
T. Tao and V. Vu, Additive combinatorics. Cambridge Studies in Advanced Mathematics, 105. Cambridge University Press, Cambridge, 2006.

\bibitem{T}
T. Tao, Recent Progress on the Restriction Conjecture. Fourier analysis and convexity, 217-243, Appl. Numer. Harmon. Anal., Birkhuser Boston, Boston, MA, 2004.

\end{thebibliography}
\end{document}